\documentclass[10pt]{amsart}
\usepackage{amssymb}
\usepackage{epsfig}
\usepackage{url}
\usepackage{setspace}
\usepackage{color}
\theoremstyle{plain}

\newtheorem{thm}{Theorem}[section]
\newtheorem{cor}[thm]{Corollary}
\newtheorem{lem}[thm]{Lemma}

\newtheorem{rem}[thm]{Remark}

\newtheorem{exam}[thm]{Example}
\def\cal{\mathcal}
\def\bbb{\mathbb}
\def\op{\operatorname}
\renewcommand{\phi}{\varphi}
\newcommand{\R}{\bbb{R}}
\newcommand{\N}{\bbb{N}}
\newcommand{\Z}{\bbb{Z}}
\newcommand{\Q}{\bbb{Q}}

\newcommand{\al}{\alpha}
\newcommand{\be}{\beta}
\newcommand{\ga}{\gamma}

\newcommand{\Ga}{\Gamma}

\begin{document}
\title[Diagonal equations]{On certain diophantine equations of diagonal type}
\author{Andrew Bremner and Maciej Ulas}

\begin{abstract}
In this note we consider Diophantine equations of the form
\begin{equation*}
a(x^p-y^q) = b(z^r-w^s), \quad
\mbox{where}\quad \frac{1}{p}+\frac{1}{q}+\frac{1}{r}+\frac{1}{s}=1,
\end{equation*}
with even positive integers $p,q,r,s$. We show that in each case the set of rational
points on the underlying surface is dense in the Zariski topology. For the surface with
$(p,q,r,s)=(2,6,6,6)$ we prove density of rational points in the Euclidean topology.
Moreover, in this case we construct infinitely many parametric solutions in coprime
polynomials. The same result is true for  $(p,q,r,s)\in\{(2,4,8,8), (2,8,4,8)\}$.
In the case $(p,q,r,s)=(4,4,4,4)$, we present some new parametric solutions of the
equation $x^4-y^4=4(z^4-w^4)$.
\end{abstract}

\maketitle

\section{Introduction}\label{sec1}

It is well known that the Diophantine equation $x^4+y^4=z^4+w^4$ has
infinitely many integer solutions (Euler). In fact, as was proved by
Swinnerton-Dyer, this diophantine equation has infinitely many
rational parametric solutions \cite{SwD}. This implies that the set
of rational points on the surface defined by this equation is dense
in all real points. The same is true for the equation
$x^4+y^4+z^4=w^4$ as was proved by Elkies in \cite{Elk}. It is also
classically known that the diophantine equation $x^6+y^6-z^6=w^2$ has infinitely
many integer solutions. According to Dickson \cite{Dic} this result was obtained
by Rignaux by construction of two polynomial
solutions. In a recent paper we extended these
investigations to equations
\begin{equation}\label{pqrsequation}
x^p \pm y^q = \pm z^r \pm w^s, \quad p,q,r,s \in \N_{+},\quad
\mbox{where}\quad \frac{1}{p}+\frac{1}{q}+\frac{1}{r}+\frac{1}{s}=1.
\end{equation}
We call the above equation an {\it equation of diagonal type}.  In \cite{BreUl} we
considered forty one equations (those not considered by Euler, Rignaux and Elkies)
corresponding to configurations of signs and exponents. We showed that for all but
one instance the existence of real solutions implies the existence of infinitely many
rational solutions.
The remaining equation which we were unable to treat is $w^2=-x^6-y^6+z^6$.  In the range
$x+y+z<5000$ we know only one numerical solution $(x,y,z,w)=(28,44,57,162967)$ and Noam
Elkies~\cite{Elk2} has informed us that this is the only solution with $x, y\leq 2^{15}$.

In this paper we study the natural generalization of the equation (\ref{pqrsequation})
in the form
\begin{equation}\label{genequation}
a(x^p-y^q) = b(z^r-w^s), \quad p,q,r,s \in \N_{+},\quad
\mbox{where}\quad \frac{1}{p}+\frac{1}{q}+\frac{1}{r}+\frac{1}{s}=1.
\end{equation}
Here $a, b$ are fixed nonzero integer numbers. It is clear that we
can assume that $\gcd(a,b)=1$. The aim of this paper is to show that
for each of the quadruplets $(p,q,r,s)$ with $p,q,r,s$ even, the
diophantine equation (\ref{genequation}) has infinitely many
rational parametric solutions.
Of course, there are trivial parametric solutions of (\ref{genequation}) corresponding to $x^p=y^q$, $z^r=w^s$,
etc., representing lines and curves on the surface in weighted projective space; but these
solutions are not of interest to us, and henceforth we discard such possibilities.

In Section \ref{sec2} we generalize
Rignaux's result by proving that when $(p,q,r,s)=(2,6,6,6)$, then there
are infinitely many polynomial solutions to (\ref{genequation}) in coprime polynomials $x,y,z,w$.
A similar result is proved in Section \ref{sec3} for $(p,q,r,s)=(2,4,8,8)$, $(2,8,4,8)$.
In Section \ref{sec4} we investigate the three equations (\ref{genequation}) that arise from
permutations of $(p,q,r,s)=(2,4,6,12)$. We show there are infinitely many coprime
polynomial solutions of (\ref{genequation}) for $(p,q,r,s)=(2,4,6,12)$, and show there
are infinitely many polynomial solutions, not necessarily coprime, for
$(p,q,r,s)=(2,6,4,12)$, $(2,12,4,6)$.  The  basic idea in these three sections is to use
an elliptic fibration on the corresponding surface in weighted projective space. In
Section \ref{sec5} we consider the case $(p,q,r,s)=(4,4,4,4)$, which needs different
techniques from elementary arithmetic algebraic geometry. It devolves into cases where
$a/b$ is not a square; $a/b$ is a square mod fourth powers, not 1 or 4; and $a/b$
mod fourth powers equals 1 or 4. In the last two sections we consider some generalizations
of the preceding results.

\section{The equation $a(x^2-y^6)=b(z^6-w^6)$}\label{sec2}

In this section we construct polynomial solutions of the equation
\begin{equation}\label{eq2666}
a(x^2-y^6)=b(z^6-w^6),
\end{equation}
where $a,b$ are coprime integers. The following theorem generalizes the result of
Rignaux~\cite{Rig}.
\begin{thm}\label{thm1}
Fix $a,b\in\Z\setminus\{0\}$. The diophantine equation
$a(x^2-y^6)=b(z^6-w^6)$ has infinitely many solutions in coprime polynomials $x,y,z,w \in \Z[t]$.
Moreover, the set of rational points on the surface $S:\;a(x^2-y^6)=b(z^6-w^6)$ is dense
in the Euclidean topology.
\end{thm}
\begin{proof}
Set $x=y^3+t(z^3-w^3)$, for an indeterminate $t$ assumed throughout to be non-zero. Then
\begin{equation*}
a(x^2-y^6)-b(z^6-w^6)=(z^3-w^3)(2a t y^3 + (a t^2-b)z^3-(a t^2+b)w^3),
\end{equation*}
and the problem reduces to investigation of the smooth cubic curve
\begin{equation*}
\cal{C}:\; 2 a t y^3 = -(a t^2-b) z^3 + (a t^2+b) w^3.
\end{equation*}
Since the equation
defining the curve $\cal{C}$ is homogeneous, a point $(y:z:w) \in \cal{C}(\Q(t))$
may be assumed to have coordinates which are polynomials in $t$. Localizing at $t=2$
results in $4 a y_2^3 = -(4a-b) z_2^3 + (4a+b) w_2^3$, which cannot have solutions
in $\Q(a,b)$ since at $a=b=1$ the equation becomes $3 z_2^3 + 4 y_2^3 + 5(-w_2)^3=0$
over $\Q$, the famous Selmer cubic with no non-zero rational solution. Thus the set
$\cal{C}(\Q(t))$ is empty for generic $a,b$.

Consider the non-invertible change of variable
given by $\phi: \Q \rightarrow \Q$ where $\phi(t)=t^3$; note that the set $\phi(\Q)$
is dense in $\Q$, which property will be used later. The corresponding curve $\cal{C}_\phi$
takes the form
\[ \cal{C}_\phi:\; 2 a t^3 y^3 = -(a t^6-b) z^3 + (a t^6+b) w^3, \]
containing the point $P(y,z,w)=(t,-1,1)$, and hence is an elliptic curve over $\Q(t)$.
The tangent to $\cal{C}_\phi$ at $P$ meets the curve again at
$Q(y,z,w)=(2 b t, 3 a t^6+b, 3 a t^6-b)$, and
\begin{equation}
\label{specsol}
(x,y,z,w) = (2 b t^3 (27 a^2 t^{12}+ 5 b^2), \; 2 b t, \; 3 a t^6+b, \; 3 a t^6-b)
\end{equation}
shows us that the equation (\ref{eq2666}) has infinitely many solutions
in coprime integers.
To prove that there are infinitely many polynomial solutions, note that with $P$ as the origin
of the group law, then $\cal{C}_\phi$ is birationally equivalent to the elliptic cubic with
Weierstrass equation
\begin{equation*}
\cal{E}:\; Y^2=X^3-27a^2(at^6-b)^2(at^6+b)^2.
\end{equation*}
The image of $Q$ on $\cal{E}$ is the point
\begin{equation}
\label{Rorder}
R=((3a^2 t^{12} + b^2)/t^4,  b(-9a^2 t^{12} + b^2)/t^6 );
\end{equation}
and by characterization of the torsion points on curves of type $y^2=x^3+D$ (see,
for example, Silverman~\cite[p. 323]{Sil}, we have that $R$ is of infinite order.
The points $mR$, for $m=2,3,4,\ldots$, pull back to points $(y_m,z_m,w_m)$ on $\cal{C}_\phi$,
with $\gcd(y_m,z_m,w_m)=1$, and then $(x_m,y_m,z_m,w_m)=(y_m^3+t^3(z_m^3-w_m^3),y_m,z_m,w_m)$
gives infinitely many coprime polynomial solutions of (\ref{eq2666}).

We will prove that the set of rational points on the surface $S:\;a(x^2-y^6)=b(z^6-w^6)$ is dense in the Euclidean topology.
However, we first prove Zariski density of the set of rational points. It is clear that in order to prove this result it is enough to prove that the set of rational points is dense in the Zariski topology on $\cal{C}_{\phi}$. Because the curve $\cal{E}$ is of
positive rank over $\Q(t)$, the set of multiples of the point $R$, i.e. $mR=(X_{m}(t),\;Y_{m}(t))$ for $m=1,2,\ldots$\;, gives
infinitely many $\Q(t)$-rational points on the curve $\cal{E}$. Now, regarding $\cal{E}$ as an elliptic surface in the
space with coordinates $(X,Y,t)$ we see that each rational curve $(X_{m},Y_{m},t)$ is included in the Zariski closure, say $\cal{R}$,
of the set of rational points on $\cal{E}$. Because this closure consists of only finitely many components, it has dimension two, and
as the surface $\cal{E}$ is irreducible, $\cal{R}$ is the whole surface. Thus the set of rational points on $\cal{E}$ is dense in
the Zariski topology and the same is true for $\cal{C}_{\phi}$ and thus for $\cal{C}$ and $S$. This follows from the fact that $\cal{C}_{\phi}$ comes from $\cal{C}\simeq S$ after a non-invertible change of variable.

To obtain the density of the set $\cal{E}(\Q)$ in the Euclidean topology, we use two
results:  a theorem of Hurwitz \cite{Hur} (see also \cite[p. 78]{Sko}) and a theorem of
Silverman \cite[p. 368]{Sil}. The theorem of Hurwitz states that if an elliptic curve $E$ defined over $\Q$ has positive
rank and at most one torsion point of order two (defined over $\Q$) then the
set $E(\Q)$ is dense in $E(\R)$. The same result holds if $E$ has three torsion points of order two under the assumption that we have
a rational point of infinite order on the bounded branch of the set $E(\R)$.

Silverman's theorem states that if $\cal{E}$ is an elliptic curve defined over $\Q(t)$ with positive rank, then for all but finitely
many $t_{0}\in\Q$, the curve $\cal{E}_{t_{0}}$ obtained from the curve $\cal{E}$ by specialization at $t=t_{0}$ has positive rank.
From this result we see that for all but finitely many $t\in\Q$ the elliptic curve $\cal{E}_{t}$ is of positive rank.
Let us denote by $\cal{G}$ the set of $t\in\Q$ such that the specialization $R_{t}$ of the point $R$ at $t$ is of finite order on the curve $\cal{E}_{t}$. From the remark at the beginning of this section we know that the order of a torsion point on the curve $\cal{E}_{t}$ is at most six. Thus, in order to find $\cal{G}$ it is enough to find all $t\in\Q$ such that $R_{t}$ has finite order $\leq 6$, that is, we need to characterize all $t\in\Q$ such that $mR_{t}=\cal{O}$ for some $m\in\{1,2,3,4,6\}$.
The assumption on $a, b$ immediately implies that there are at most two specializations of $t\in\Q$
such that the $Y$ coordinate of the point $R$ is equal to zero, i.e. the order of the point $R$ is two,
and these specializations correspond to the rational roots of the equation $(3at^6-b)(3at^6+b)=0$.
The vanishing of the $3$-division polynomial of $\cal{E}$ evaluated at $X(R_{t})$ determines
a polynomial of degree $4$ in $t^{12}$, so there can be at most $8$ rational roots $t$, and so at most
$8$ values of $t\in\Q$ leading to $R_{t}$ of order three. Similarly, the vanishing of the $4$-division
polynomial at $X(R_{t})$ determines a polynomial of degree $8$ in $t^{12}$, so there are at most
$16$ rational values of $t$ leading to $R_{t}$ of order four. Finally, a point of order six can occur
if and only if $-27a^2(at^6-b)^2(at^6+b)^2$ equals a sixth power, which is clearly impossible.
To sum up, there are at most $2+8+16=26$ rational specializations of $t$ such that the point
$R_{t}$ has finite order, and thus $|G|\leq 26$ (note also that $t=0$ is forbidden).

Using the Silverman theorem we deduce that for all $t\in\Q\setminus G$ the curve $\cal{E}_{t}$ is of
positive rank. The Hurwitz theorem now implies that the set $\cal{E}_{t}(\Q)$ is dense in $\cal{E}_{t}(\R)$.
It follows that $\cal{C}_\phi(t)(\Q)$ is dense in $\cal{C}_\phi(t)(\R)$ because the image of $\Q$
by the function $\phi(t)=t^3$ is dense in the Euclidean topology of $\R$ (since $\overline{\phi(\Q)}=\R$).
If we put now $H=\phi(\Q)\subset \Q$ we get that for all but finitely many $t\in H$ the set $\cal{C}_{t}(\Q)$
is dense in $\cal{C}_{t}(\R)$. Because $H$ is dense in $\R$, it follows that the set $S(\Q)$ of rational points on $S$ is dense in the Euclidean topology in $S(\R)$. The theorem is proved.
\end{proof}
\begin{cor}\label{cor2}
Fix $a,b\in\Z\setminus\{0\}$ and $n\in\N_{+}$. The diophantine equation
\begin{equation}\label{eq26n66}
a(x^2-y^{6n})=b(z^6-w^6)
\end{equation}
has infinitely many solutions in integers. Moreover, if $b=1$, then
the equation has infinitely many solutions in coprime integers.
\end{cor}
\begin{proof}
Substituting $t=(2b)^{n-1}T^n$ into the expression for $x,y,z,w$ in
(\ref{specsol}) we get a polynomial solution of (\ref{eq26n66}).
Unfortunately this solution has the common factor $b$. However,
if $b=1$ then we get a solution of the equation
$a(x^2-y^{6n})=z^6-w^6$ in the following form:
\begin{equation*}
\begin{array}{ll}
  x=2^{3n-2}T^{3n}(27a^22^{12(n-1)}T^{12n}+5), & y=2T, \\
                                              & \\
  z=3a2^{6(n-1)}T^{6n}+1, & w=3a2^{6(n-1)}T^{6n}-1.
\end{array}
\end{equation*}
This solution clearly satisfies the condition $1=\gcd(z,w)=\gcd(x,y,z,w)$.
\end{proof}

\section{The equation $a(x^{p}-y^{q})=b(z^{r}-w^{s})$ with $(p,q,r,s)=(2,4,8,8),(2,8,4,8)$}\label{sec3}

In this section we construct polynomial solutions of the equation
\begin{equation}\label{eq2488}
a(x^p-y^q)=b(z^r-w^s),
\end{equation}
where $a,b$ are coprime integers and
$(p,q,r,s)\in\{(2,4,8,8),(2,8,4,8)\}$.

First, a simple lemma.
\begin{lem}\label{lem1}
Consider the elliptic quartic $E: \; y^2=\al x^4+\be$, $\al+\be=\ga^2$, $\al,\be,\ga \in \Z$.
If $\al\be \neq \Box$ and $\al\be \neq -1$ mod fourth powers, then the curve has torsion group
of order 2; and  since $E$ contains the four points $(\pm 1, \pm \ga)$, then $E$ has positive rank.
\end{lem}
\begin{proof}
The curve $E$ is birationally equivalent to the Weierstrass cubic $Y^2=X(X^2-4\al\be)$,
and the result follows from characterization of torsion on curves of the form $Y^2=X(X^2+D)$:
see, for example, Silverman~\cite[p. 311]{Sil}.
\end{proof}

\begin{thm}\label{thm2}
Let $a,b\in\Z\setminus\{0\}$. If $(p,q,r,s)\in\{(2,4,8,8),\;
(2,8,4,8)\}$ then the diophantine equation
$a(x^{p}-y^{q})=b(z^{r}-w^{s})$ has infinitely many
solutions in coprime polynomials $x, y, z, w\in\Z[t]$.
\end{thm}
\begin{proof}
Consider first the case $(p,q,r,s)=(2,4,8,8)$. Set $x=-y^2+t^2(z^4+w^4)$,
for an indeterminate $t$. Then
\begin{equation*}
a(x^2-y^4)-b(z^8-w^8)=-(z^4+w^4)(2at^2y^2-(at^4-b)z^4-(at^4+b)w^4).
\end{equation*}
We prove the desired result by showing that the curve
\begin{equation}
\cal{C}: \; 2at^2y^2 = (at^4-b)z^4 + (at^4+b)w^4
\end{equation}
has infinitely many solutions in the ring of polynomials $\Z[t]$.
Note that if $y, z, w\in\Z[t]$ satisfy this equation and
$t \nmid \gcd(z,w)$ then $\gcd(z,w)^2$ divides $y$. This implies
that we can assume $\gcd(y,z,w)=1$ in $\Z[t]$. Thus the solution will
be in coprime polynomials provided $t \nmid \gcd(z,w)$.\\


Now $\cal{C}$ contains the points $(y,z,w)=(\pm t, \pm 1, 1)$. If we take
$(t,1,1)$ as the point at infinity, then by the remarks of Lemma \ref{lem1},
the point $Q=(-t,1,1)$ is of infinite order. We have
\[ 2Q(y,z,w)=(t(-3b^4+4 a^4 t^{16}), \; b^2-2 a b t^4 -2 a^2 t^8, \; b^2+2 a b t^4 -2 a^2 t^8), \]
so that
\begin{align*}
& x= t^2(-7b^8+32a^2b^6t^8-88a^4b^4t^{16}+128a^6b^2t^{24}+16a^8t^{32}),\\
& y= t(-3b^4+4a^4t^{16}), \\
& z= b^2-2abt^4-2a^2t^8, \\
& w= b^2+2abt^4-2a^2t^8,
\end{align*}
gives infinitely many solutions of our equation in coprime integers.\\
Let $mQ=(y_m,z_m,w_m)$, $m=1,2,3,\dots$. The recurrence formulae that determine $mQ$ in terms
of $(m-1)Q$ are necessarily complicated, but it may be checked by induction using
the addition formula that
\[ z_m \equiv w_m \equiv b^{m^2-m} \bmod t, \qquad y_m \equiv 0 \bmod t, \]
for $m=1,2,3,...$.  Hence $t \nmid \op{gcd}(z_m,w_m)$, and the points $mQ$, together with $x_m=-y_m^2+t^2(z_m^4+w_m^4)$,
lead to infinitely many coprime polynomial solutions.
\bigskip

The proof in the case $(p,q,r,s)=(2,8,4,8)$ is similar. This time
set $x=-y^4+t^4(z^2+w^4)$ and get
\begin{equation*}
a(x^2-y^8)-b(z^4-w^8)=-(z^2+w^4)(bz^2-bw^4+2ay^4t^4-az^2t^8-aw^4t^8).
\end{equation*}
The curve
\[ \cal{C}: \; (b-a t^8)z^2 = (b+a t^8)w^4 - 2a t^4 y^4 \]
has points $(\pm y, \pm z,w)=(t,1,1)$, and so by Lemma \ref{lem1},
on taking $(t,1,1)$ as the origin, the point $Q=(-t,1,1)$ is of infinite
order. Then $2Q$ gives
\begin{align*}
x= & t^4(-79b^8+600ab^7t^8+3068a^2b^6t^{16}+18984a^3b^5t^{24}+101126a^4b^4t^{32}\\
   & \hskip 0.5cm +155112a^5b^3t^{40}+172604a^6b^2t^{48}+109848a^7bt^{56}+28561a^8t^{64}),\\
y= & t(-3b^2+6abt^8+13a^2t^{16}),\\
z= & b^4-52ab^3t^8-138a^2b^2t^{16}-340a^3bt^{24}-239a^4t^{32},\\
w= & b^2+14abt^8+a^2t^{16}.
\end{align*}
With similar reasoning to the previous case, we easily deduce the
existence of infinitely many coprime polynomial solutions to the
equation $a(x^2-y^8)=b(z^4-w^8)$.
\end{proof}

There is an immediate consequence from the above results.
\begin{cor}
Consider the surfaces
\begin{equation*}
\cal{S}_{1}:\;a(x^2-y^4)=b(z^8-1),\quad \cal{S}_{2}:\;a(x^2-y^8)=b(z^4-1).
\end{equation*}
The set of rational points on $\cal{S}_{i}$ is Zariski dense for $i=1,2$.
\end{cor}
\begin{proof}
This follows from the existence of infinitely many rational curves lying on $\cal{S}_i$, $i=1,2$, and
the reasoning given in the proof of Theorem 2.1.
\end{proof}

\section{The equation $a(x^{p}-y^{q})=b(z^{r}-w^{s})$ with $(p,q,r,s)\in\{(2,4,6,12),(2,6,4,12),(2,12,4,6)\}$}\label{sec4}

In this section we are interested in constructing polynomial solutions of the equation
\begin{equation}\label{eq24612}
a(x^p-y^q)=b(z^r-w^s),
\end{equation}
where $a,b$ are coprime integers and
$(p,q,r,s)\in\{(2,4,6,12),(2,6,4,12),(2,12,4,6)\}$. We prove the following:

\begin{thm}\label{thm3}
Let $a,b\in\Z\setminus\{0\}$. The diophantine equation
$a(x^2-y^4)=b(z^6-w^{12})$ has infinitely many coprime
solutions in the ring of polynomials $\Z[t]$.
\end{thm}
\begin{proof}
Set $x=-y^2 + t^6(z^3+w^6)$, where $t$ is an
indeterminate, and get
\begin{equation*}
a(x^2-y^4)-b(z^6-w^{12})=-(z^3+w^6)(2a t^6y^2 - (a t^{12}-b)z^3 - (a t^{12}+b)w^6).
\end{equation*}
The problem reduces to investigating the curve
\begin{equation*}
\cal{C}: \; 2a t^6y^2=(a t^{12}-b)z^3 + (a t^{12}+b)w^6,
\end{equation*}
with Weierstrass form
\[ \cal{E}: \; Y^2 = X^3 + 8a^3(a t^{12}-b)^2(a t^{12}+b) \]
under the mapping
\[ (X,Y)=\left( \frac{2a(a t^{12}-b)z}{w^2}, \; \frac{4a^2t^3(a t^{12}-b)y}{w^3} \right). \]
There is the obvious point $P(z,w,y)=(1,1,t^3)$ mapping to $Q=(2a(a t^{12}-b), -4a^2 t^6(a t^{12}-b))$
on $\cal{E}$. Just as for the point $R$ at (\ref{Rorder}), we can show that $Q$ is of infinite order. The point $2Q$ leads to
the following solution of $a(x^2-y^4)=b(z^6-w^{12})$:
\begin{align*}
& x = -4a^2(729b^4+2430a b^3t^{12}+3024a^2b^2t^{24}+2178a^3b t^{36}-169a^4t^{48}), \\
& y = 2a(-27b^2-18a bt^{12}+13a^2t^{24}), \\
& z = -2a t^2(9b+7a t^{12}),\\
& w = 4a t^7,
\end{align*}
and in general, we construct infinitely many polynomial solutions by pulling
back multiples of $Q$. If $z,w$ are both divisible by an irreducible polynomial $f(t)$
coprime to $t(a t^{12}-b)$, then necessarily $f^2 \mid z$, $f^3 \mid y$, and we have
the polynomial solution $(z/f^2, w/f, y/f^3)$. To give infinitely many polynomial
solutions with coprime $(x,y,z,w)$ it therefore suffices to give an infinite family
where $y$ is coprime to $t(a t^{12}-b)$. Inductively, we claim
that the pull-backs of points $2^nQ$, $n=1,2,3,...$, satisfy this condition.

Let $(z_n,w_n,y_n)=2(z_{n-1},w_{n-1},y_{n-1})$ on $\cal{C}$, with $(z_0,w_0,y_0)=(1,1,t^3)$.
The duplication formula gives
\begin{align*}
z_n & =  2 a t^2 z_{n-1} \left( (a t^{12}-b) z_{n-1}^3 - 8 (a t^{12}+b) w_{n-1}^6 \right), \\
w_n & = 4 a t^4 w_{n-1} y_{n-1}, \\
y_n & = 2 a \left( -(a t^{12}-b)^2 z_{n-1}^6-20 (a t^{12}+b)(a t^{12}-b) z_{n-1}^3 w_{n-1}^6 + 8(a t^{12}+b)^2 w_{n-1}^{12} \right)
\end{align*}
Modulo $a t^{12}-b$, it follows that
\[ y_n \equiv 64 a b^{12} w_{n-1}^4 \equiv 64 a^2 b y_{n-1}^4. \]
Since $y_0=t^3$, it immediately follows that $y_n$ is coprime to $a t^{12}-b$ for $n \geq 0$.
Assume now that $t^2 || z_{n-1}$, $t^7 || w_{n-1}$, and $t \nmid y_{n-1}$; this is certainly
true in the case $n=2$ from the above particular solution. The recurrence relations above show that
$t^{10} || z_n$, $t^{11} || w_n$, $t^{12} || y_n$. Replacing $(z,w,y)$ by $(z/t^8, w/t^4, y/t^{12})$
gives a reduced polynomial solution in which $t^2 || z_n$, $t^7 || w_n$, $t \nmid y_n$, as required.
\end{proof}

\begin{rem}\label{rem26412}
{\rm Unfortunately, we are unable to find coprime polynomial
solutions of the equation $a(x^p-y^q)=b(z^r-w^s)$ with
$(p,q,r,s)=(2,6,4,12),(2,12,4,6)$. However, we do construct
polynomial solutions of these equations. First, consider the
equation $a(x^{p}-y^{q})=b(z^{r}-w^{s})$ with
$(p,q,r,s)=(2,6,4,12)$. Set $x=y^3+t^6(z^2+w^6)$ to get
\begin{equation*}
a(x^{2}-y^{6})-b(z^{4}-w^{12})=(z^2+w^6)(2a t^6 y^3 +(a t^{12}-b) z^2 + (a t^{12}+b) w^6),
\end{equation*}
and our problem is reduced to investigation of the curve
\begin{equation*}
\cal{C}: \; (a t^{12}-b)z^2 = -2a t^6 y^3 - (a t^{12}+b) w^6.
\end{equation*}
The curve $\cal{C}$ is birationally equivalent with the Weierstrass cubic:
\begin{equation*}
\cal{E}: \; Y^2=X^3 - 4a^2(a t^{12}-b)^3(a t^{12}+b),
\end{equation*}
under the mapping
\begin{equation*}
(X,Y)=\left( \frac{-2at^2(a t^{12}-b)y}{w^{2}}, \; \frac{2a(a t^{12}-b)^2 z}{w^{3}} \right).
\end{equation*}
There is an obvious point $(y,w,z)=(-t^2,1,1)$ on $\cal{C}$ mapping
to the point
\begin{equation*}
Q=(2a t^4 (a t^{12}-b), 2a(a t^{12}-b)^2),
\end{equation*}
on $\cal{E}$, which as before is of infinite order.
The point $2Q$ leads to the following solution
of the diophantine equation
$a(x^{2}-y^{6})=b(z^{4}-w^{12})$:
\begin{align*}
& x=8 t^6 (a t^{12}-b)^2 (125a^4 t^{48}+409a^3bt^{36}+588a^2b^2t^{24}+256ab^3t^{12}+80b^4), \\
& y=-2 t^2 (a t^{12}-b)(5a t^{12}+4b), \\
& z=-4(a t^{12}-b)(11a^2 t^{24}+14 a b t^{12}+2b^2),\\
& w=2(a t^{12}-b).
\end{align*}
Again, we can write down infinitely many polynomial solutions (not necessarily
coprime).

\bigskip

Second, consider the equation $a(x^2-y^{12})=b(z^4-w^6)$.
Set $x=y^6 - t^6(z^2-w^3)$, where $t$ is an indeterminate, and get
\begin{equation*}
a(x^{2}-y^{12})-b(z^{4}-w^{6})=(w^3-z^2)((a t^{12}+b) w^3 + 2a t^6 y^6 - (a t^{12}-b) z^2),
\end{equation*}
and our problem is reduced to investigating the curve
\begin{equation*}
\cal{C}: \; (a t^{12}-b)z^2=(a t^{12}+b) w^3 + 2a t^6 y^6.
\end{equation*}
The curve $\cal{C}$ is birationally equivalent with the Weierstrass cubic
\begin{equation*}
\cal{E}: \; Y^2 = X^3 + 2a(a t^{12}-b)^3 (a t^{12}+b)^2
\end{equation*}
under the mapping
\begin{equation*}
(X,Y)=\left( \frac{(a t^{12}+b) (a t^{12}-b)w}{t^{2}y^{2}}, \; \frac{(a t^{12}-b)^2 (a t^{12}+b)z}{t^3y^{3}} \right).
\end{equation*}
The obvious point $(w,y,z)=(-1,t,1)$ on $\cal{C}$ maps to
\begin{equation*}
Q=\left( \frac{-(a t^{12}-b)(a t^{12}+b)}{t^4}, \;  \frac{-(a t^{12}-b)^2(a t^{12}+b)}{t^6} \right)
\end{equation*}
on $\cal{E}$, and as before is of infinite order.
The point $2Q$ leads to the following solution
of the diophantine equation $a(x^2-y^{12})=b(z^4-w^6)$:
\begin{align*}
& x = -2t^6(a t^{12}-b)^2 (32a^4 t^{48}+4849a^3 b t^{36}+867a^2 b^2 t^{24}+115a b^3 t^{12}-31b^4), \\
& y = 2t (a t^{12}-b), \\
& z = (a t^{12}-b)(-71a^2 t^{24}-38a b t^{12}+b^2),\\
& w = (a t^{12}-b)(17a t^{12}+b).
\end{align*}
Infinitely many (not necessarily coprime) polynomial solutions are constructed
by taking the pullbacks of $mQ$, $m=2,3,4,\dots$.
}
\end{rem}
As a consequence of Theorem \ref{thm3}, Remark \ref{rem26412}, and the reasoning as in Corollary 3.4,
we have the following.
\begin{cor}
Consider the surfaces
\begin{equation*}
\cal{S}_{3}:\;a(x^2-y^4)=b(z^6-1),\quad \cal{S}_{4}:\;a(x^2-y^6)=b(z^4-1),\quad \cal{S}_{5}:\;a(x^2-1)=b(z^4-w^6)
\end{equation*}
The set of rational points on $\cal{S}_{i}$ is Zariski dense for $i=3,4,5$.
\end{cor}

\section{The equation $a(x^4-y^4)=b(z^4-w^4)$}\label{sec5}
We take the equation in the form
\[ V: \; x^4-y^4=h(z^4-w^4), \]
representing a surface in projective three-space.
Choudhry~\cite{Ch} has some elementary results showing how to derive new points from
known points, essentially arising from elliptic fibrations of the surface of type
\[ x^2-y^2 = t(z^2-w^2), \quad t(x^2+y^2) = h(z^2+w^2), \]
where the known point means the intersection of the two quadrics is an elliptic curve.\\ \\
We ask whether there exist parametrizable curves on $V$, that is, curves of geometric
genus 0. We are only able to treat the case of curves that have arithmetic genus 0,
hence geometric genus 0, and hence parametrizable.
The arguments we use are those of Swinnerton-Dyer~\cite{SwD}. \\ \\
It is known (see for example Pinch \& Swinnerton-Dyer~\cite{P-SD}) that there are precisely
48 straight lines on the surface $V$, given as follows:
\begin{align*}
& (x=\alpha_1 y, z=\beta_1 w), \quad  \alpha_1^4=\beta_1^4=1,  \\
& (x=\alpha_2 z, y=\beta_2 w), \quad  \alpha_2^4=\beta_2^4=h,  \\
& (x=\alpha_3 w, y=\beta_3 z), \quad  \alpha_3^4=\beta_1^4=-h;
\end{align*}
further, the  N\'eron-Severi group of $V$ over $\mathbb{C}$ is generated by the classes
of these lines. \\
Write $h=\theta^4$.  The hyperplane $x-y-\theta(z-w)=0$ cuts the surface $V$ in the lines
$(x=y,z=w)$, $(x=\theta z, y=\theta w)$, with residual intersection the irreducible conic
\[ x^2 - x y + 2 y^2 - \theta x z + 3 \theta y z + 2 \theta^2 z^2=0. \]
In this way we generate 128 distinct irreducible conics on $V$, typified by the above. \\ \\
In the first instance, we suppose $h$ is not a perfect rational square. Pinch \& Swinnerton-Dyer
show that the N\'eron-Severi group of $V$ over $\mathbb{Q}$ is of rank 6. It is
straightforward to show that the group has a $\Z$-basis given by (the classes of) the following four
lines and two line pairs:
\begin{align*}
\Delta_1=(x=y,z=w), \qquad  & \Delta_2=(x=y,z=-w), \\
\Delta_3=(x=-y,z=w), \qquad & \Delta_4=(x=-y,z=-w), \\
\Delta_5=(x=y,z=i w)+(x=y,z=-i w), \quad & \Delta_6=(x=i y,z=i w)+(x=-i y,z=-i w).
\end{align*}
The corresponding intersection matrix is given in Table 1.  \\
\begin{table}[h]
\caption{Intersection matrix for a $\Z$-basis of $V$ with $h$ not a square}
\begin{center}
$\begin{array}{c||c|c|c|c|c|c|}
 & \Delta_1 & \Delta_2 & \Delta_3 & \Delta_4 & \Delta_5 & \Delta_6 \\ \hline \hline
\Delta_1 &-2 & 1 & 1 & 0 & 2 & 0 \\ \hline
\Delta_2 & 1 & -2 & 0 & 1 & 2 & 0 \\ \hline
\Delta_3 & 1 & 0 & -2 & 1 & 0 & 0 \\ \hline
\Delta_4 & 0 & 1 & 1 & -2 & 0 & 0 \\ \hline
\Delta_5 & 2 & 2 & 0 & 0 & -2 & 2 \\ \hline
\Delta_6 & 0 & 0 & 0 & 0 & 2 & -4 \\ \hline
\end{array}$
\end{center}
\end{table}
For a curve $\Gamma$ defined over $\mathbb{Q}$, set
\[ \Ga \sim n_1 \Delta_1+n_2 \Delta_2+n_3 \Delta_3+n_4 \Delta_4+n_5 \Delta_5+n_6 \Delta_6. \]
Defining $g(\Gamma)$ to be the arithmetic genus of $\Gamma$, we have
\begin{align*}
2 \; g(\Ga) - 2 = (\Ga \cdot \Ga) = & -2 n_1^2 + 2 n_1 n_2 + 2 n_1 n_3 + 4 n_1 n_5 - 2 n_2^2 + 2 n_2 n_4 + 4 n_2 n_5 \\
& - 2 n_3^2 + 2 n_3 n_4 - 2 n_4^2 - 2 n_5^2 + 4 n_5 n_6 - 4 n_6^2.
\end{align*}
Writing
\[ \mbox{deg}(\Ga) = d = n_1+n_2+n_3+n_4+2n_5+2n_6, \]
we have
\begin{align*}
d^2-4(\Ga.\Ga) = & (-d + 4 n_5)^2 + 4 (-d + n_2 + 3 n_5 + n_3 + 2 n_4 + 2 n_6)^2\\
& +4 (-d + 2 n_2 + 2 n_5 + 2 n_3 + 2 n_6)^2 + 4 (n_2-n_5-n_3)^2 + 16 n_6^2,
\end{align*}
allowing efficient computation of divisors over ${\Q}$ of given degree
and genus. \\ \\
Suppose that $\Gamma$ is irreducible and is distinct from any of the
known lines or conics. Then $\Gamma$ will have non-negative intersection
number with all the straight lines and conics, and this determines linear inequalities
on the coefficients $n_i$, $i=1,...,6$. In this way, we obtain 31 linear constraints
to be satisfied by the $n_i$, and accordingly $(n_1,n_2,n_3,n_4,n_5,n_6)$ is a point
of the corresponding convex cone. Supporting hyperplanes of this cone
are $22$ in number (we used the routine {\tt MinimalInequalities} in Magma~\cite{Mag}).
Now
\[ g(\Ga) - 1 = \frac{1}{2} (\Ga \cdot \Ga) = \frac{1}{8} d^2 - (\mbox{positive definite form in the $n_i$}), \]
and we shall see that $\frac{1}{8} d^2 - (\mbox{pos. def. form in the $n_i$})$ is
non-negative on this reduced cone, forcing $g(\Ga) \geq 1$.
On the hyperplane $d=\mbox{const}$, this latter form takes its minimum at a vertex of the
resulting convex polytope, and so it is enough to see that the form, equivalently,
$(\Ga \cdot \Ga)$, is non-negative on the extremal rays of the cone (that is, the
lines joining the origin to the vertices of the convex polytope in which any hyperplane
meets the cone). In this particular instance, there are 41 such rays, and the minimum
value taken by $(\Ga \cdot \Ga)$ is 0.  It follows that there are no curves $\Gamma$
on $V$ of (arithmetic) genus 0 other than lines and conics.  \\ \\
Suppose second that $h$ is a perfect rational square, but not equal to $1$ or $4$ modulo fourth powers.
In this case the N\'eron-Severi group of $V$ over $\mathbb{Q}$ is of rank 7, with basis
as above together with the extra divisor
\[ \Delta_7 = (x=\theta z, y=\theta w) + (x=-\theta z, y=-\theta w). \]
Arguing as before, an irreducible curve
$\displaystyle\sum\limits_{i=1}^7 n_i \Delta_i$,
distinct from the straight lines and known conics, determines a point $(n_1,...,n_7)$
that lies within a polytope defined by the 49 half-planes arising from demanding
non-negative intersection of $\Gamma$ with the known lines and conics.
The reduced cone is defined by 30 half-planes, with 113 extremal rays. Again the
minimum value of $(\Ga \cdot \Ga)$ on the extremal rays is equal to 0, and it follows
that there are no rational curves $\Gamma$ on $V$ of (arithmetic) genus 0 other than lines
and conics. \\ \\
Third, if $h$ is $4$ modulo fourth powers, then without loss of generality, $h=4$.
This surface appears in the literature (Choudhry\cite{Ch})), who presents parameterizations
of degrees 3 and 13, but a full treatment seems not to have been given. The surface
has N\'eron-Severi rank 9, with basis as above, together with the two additional line pairs
\begin{align*}
\Delta_8 & = (x=(1+i)w, y=(1+i)z) + (x=(1-i)w, y=(1-i)z), \\
\Delta_9 & = (x=(1+i)w, y=-(1+i)z) + (x=(1-i)w, y=-(1-i)z).
\end{align*}
Using similar analysis as before, we can show that there is up to symmetry a unique curve
of degree 3, parameterized by:
\[ x:y:z:w = 4-2t+4t^2+t^3 \; : \; 4+2t+4t^2-t^3 \; : \; 2-4t-t^2-t^3 \; : \; 2+4t-t^2+t^3. \]
There are no curves of degree 5, and up to symmetry a unique curve of degree 7:
\begin{align*}
x:y:z:w = & -64 - 24t - 24t^2 + 132t^3 - 144t^4 + 138t^5 - 22t^6 + 9t^7: \\
& -96 + 8t - 24t^2 - 252t^3 + 168t^4 - 78t^5 + 42t^6 - 7t^7: \\
& -56 + 168t - 156t^2 + 168t^3 - 126t^4 - 6t^5 + t^6 - 6t^7: \\
& -72 + 88t - 276t^2 + 144t^3 - 66t^4 + 6t^5 + 3t^6 + 4t^7.
\end{align*}
It seems plausible that there are no rational curves of even degree, and curves of every
odd degree at least 7: (with some effort) we have written down such for all odd degrees up
to 25.
The techniques of Swinnerton-Dyer should resolve this question, in that all curves
should be realisable by repeated application to the straight lines of a finite set
of automorphisms of the surface. We have not taken this further here.  \\ \\
Finally, if $h$ is a perfect fourth power, then without loss of generality, $h=1$.
The surface was known by Euler to contain parametrizable curves, and has been fully treated by Swinnerton-Dyer~\cite{SwD}

\section{The equation $a(y_{1}^4-f_{1}(X)^{2})=b(y_{2}^4-f_{2}(X)^2)$ }\label{sec6}

The same ideas as used in the preceding sections allow treatment of slightly
more general equations.  In a recent paper \cite{Ul}, {\it inter alia}, it is proved that
for any pair of non-zero integers $a,b$ and any positive odd $n$ the
diophantine equation
\begin{equation}\label{prev}
a(y_{1}^{4}-x_{1}^{2n})=b(y_{2}^{4}-x_{2}^{2n}),
\end{equation}
has infinitely many rational parametric solutions. In order to get
this result the variety defined by the equation
(\ref{prev}) is treated as an intersection of two rational hypersurfaces
defined over the field $\Q(t)$. Using a suitable
(non-invertible) change of variables the study of the intersection
is reduced to the problem of constructing
$\Q(t)$-rational points on a certain hyperelliptic quartic curve, say
$\cal{C}$. An important feature of this construction is that the
curve $\cal{C}$ has a $\Q(t)$-rational point at infinity. Thus,
essentially we can transform $\cal{C}$ into an elliptic curve, say
$\cal{E}$, with Weierstrass equation defined over the field $\Q(t)$.
The existence of another $\Q(t)$-point on $\cal{C}$ guarantees that $\cal{E}$
has positive $\Q(t)$-rank, implying that the set of
$\Q(t)$-rational points on $\cal{E}$, and thus on $\cal{C}$, is
infinite. From this it follows that the set of parametric solutions
of (\ref{prev}) is infinite. This method is very simple and in
essence is very similar to the method used in previous
sections here. A natural question arises as to whether we can use a corresponding
approach to other diophantine equations of similar nature.

In this section we are interested in the following generalization
of a result related to the solvability of the diophantine equation
(\ref{prev}).

\begin{thm}\label{thmgen2}
Let $a,b$ be non-zero integers and let $f_{1}(\overline{X})$,
$f_{2}(\overline{X})$ be homogenous forms with integer
coefficients, where $\overline{X}=(X_{1},X_{2},\ldots,X_{n})$ is
a vector of variables. Suppose that $\op{deg}f_{1}=\op{deg}f_{2}=2m+1$,
$m \in \Z$. Moreover, suppose that the set of rational
points on the variety
\begin{equation*}
\cal{H}:\;Y^2=-f_{1}(\overline{X})f_{2}(\overline{X})
\end{equation*}
is infinite. Then there are infinitely many rational points lying on the hypersurface
defined by the equation
\begin{equation}\label{eq3}
\cal{V}:\;a(y_{1}^4-f_{1}(\overline{X})^{2})=b(y_{2}^4-f_{2}(\overline{X})^2).
\end{equation}
If $\cal{H}$ contains rational curves, then so does $\cal{V}$.
\end{thm}
\begin{proof}
Without loss of generality we can assume that $\op{gcd}(a,b)=1$.
Instead of considering $\cal{V}$, consider the variety defined by the intersection
\begin{equation}\label{sys2}
a(y_{1}^2-f_{1}(\overline{X}))=b U(y_{2}^2-f_{2}(\overline{X})),\quad
U(y_{1}^2+f_{1}(\overline{X}))=y_{2}^2+f_{2}(\overline{X}),
\end{equation}
where $U$ is an indeterminate parameter. In order to solve the above system, take
\begin{equation}\label{sub2}
y_{1}=T^{m},\quad y_{2}=vT^{m},\quad X_{i}=u_{i}T \quad
\mbox{for}\quad\;i=1,2,\ldots,n.
\end{equation}
From the first equation in (\ref{sys2}) we get (after clearing
the common factor $T^{2m}$)
\begin{equation}\label{expressionforT}
T=\frac{a-b U v^2}{a f_{1}-b f_{2} U},
\end{equation}
where to shorten the notation put
$f_{i}=f_{i}(\overline{u})$, with
$\overline{u}=(u_{1},\ldots,u_{n})$ and $i=1,2$. From the second
equation in (\ref{sys2}):
\begin{equation*}
(b f_{1} U^2-2b f_{2}U+a f_{1})v^2=-b f_{2} U^2+2 a f_{1} U-a f_{2}.
\end{equation*}
Thus we get the equation of a hyperelliptic quartic curve in the form
\begin{equation*}
\cal{C}_{a,b}:\;V^2=(b t U^2+2a U+a t)(b U^2+2b t U+a),
\end{equation*}
where $V=v(b f_{1} U^2-2b f_{2} U+a f_{1})/f_{1}$ and
$t=t(\overline{u})=-f_{2}(\overline{u})/f_{1}(\overline{u})$. From
the assumption which says that the variety
$Y^2=-f_{1}(\overline{u})f_{2}(\overline{u})$ has infinitely many
rational points we know that for infinitely many $n$-tuples
$\overline{u}$ of rational numbers the value of $t(\overline{u})$ is
a square. We can treat the curve $\cal{C}_{a,b}$ as the curve
defined over the function field $\Q(\cal{H})$ of the variety
$\cal{H}$.

Note that when $t(\overline{u})$ is square, then the curve $\cal{C}_{a,b}$ possesses the
$\Q(\cal{H})$-rational point $P=(0,a \sqrt{t(\overline{u})})$. Taking the point $P$ as a
point at infinity on the curve $\cal{C}_{a,b}$, then $\cal{C}_{a,b}$ is
birationally equivalent to the elliptic curve with Weierstrass
equation
\begin{equation*}
\cal{E}_{a,b}:\;Y^2=X^3+4ab(a-bt^4)^2X.
\end{equation*}
The curve $\cal{E}_{a,b}$ contains the rational point
$Q=(X,Y)$, where
\begin{align*}
&X=\frac{(a^2-6abt^4+b^2t^8)^2}{4t^2(a+bt^4)^2},\\
&Y=\frac{(a^2-6abt^4+b^2t^8)(a^4+20a^3bt^4-26a^2b^2t^8+20ab^3t^{12}+b^4t^{16})}{8t^3(a+bt^4)^3}.
\end{align*}
As before, it is easy to see that $Q$ is of infinite order in the
group $\cal{E}_{a,b}(\Q(\cal{H}))$. Indeed, from the infinitude of
the rational points on the variety $\cal{H}$ one can find a
$n$-tuple $\overline{u}$ such that $t$ is finite and non-zero.
Moreover the point $Q$ is non-trivial, i.e., $XY\neq 0$, and
$4ab(a-bt^4)^2\neq 4$ and it is not a fourth power. This implies
that the corresponding $Q$ is of infinite order. Summing up, we see that for
any $k\in\N_{+}$ the point $kQ=(X_{k},Y_{k})$ allows the computation
of values of $v$ and $T$ and thus leads to a rational point on the
hypersurface $\cal{V}$ from the expressions (\ref{sub2}).
Because $Q$ is of infinite order we get infinitely many rational
points on $\cal{V}$. The argument is analogous for when $\cal{H}$ contains
rational curves. This finishes the proof of the theorem.
\end{proof}

We note an interesting result that follows from this theorem.

\begin{cor}
Consider $n$-forms
$f_{i}(\overline{X})=L_{i}(\overline{X})F_{i}(\overline{X})^{2}$,
where $L_{i}$ is a linear $n$-form for $i=1,2$ and $L_{1}, L_{2}$
are independent. Then the hypersurface $\cal{V}$ given by the equation (\ref{eq3})
has infinitely many rational parametric solutions depending on $n-1$ variables.
\end{cor}
\begin{proof}
It is clear that the result will follow if we
show that the hypersurface $\cal{H}$ given in the preceding theorem
contains a rational hypersurface parameterized by rational functions
depending on $n-1$ parameters. In our situation $\cal{H}$ takes the
form
\begin{equation*}
\cal{H}:\;Y^2=-L_{1}(\overline{X})L_{2}(\overline{X})(F_{1}(\overline{X})F_{2}(\overline{X}))^2.
\end{equation*}
$\cal{H}$ is clearly rational. Indeed, a parameterization can be
obtained easily by solving the system
$L_{1}(\overline{X})=U_{1}^2,\;L_{2}(\overline{X})=-U_{2}^2$
(which clearly has solution by independence of $L_1$ and $L_2$) and putting
$Y=U_{1}U_{2}F_{1}(\overline{X})F_{2}(\overline{X})$.
\end{proof}

We also prove the following result of independent interest.

\begin{thm}\label{thmgen1}
Let $a,b$ be non-zero fixed integers, $n\in\N_{+}$ and put $\overline{X}=(x_{1},\ldots,x_{n})$.
Let $f_{1}, f_{2}$, be homogenous forms with integer coefficients
and suppose that $\op{deg}f_{2}=2\op{deg}f_{1}+1$. Then the set of
rational points on the hypersurface defined by the equation
\begin{equation*}\label{eq1}
\cal{W}:\;a(y_{1}^4-f_{1}(\overline{X})^{4})=b(y_{2}^4-f_{2}(\overline{X})^2)
\end{equation*}
is dense in the Zariski topology. Moreover, there are infinitely many rational curves
which lie on $\cal{W}$.
\end{thm}
\begin{proof}
Put $\op{deg}f_{1}=m$ and $\op{deg}f_{2}=2m+1$. Without loss
of generality, $\op{gcd}(a,b)=1$. In order to construct rational points on the hypersurface
$\cal{W}$ we make the following change of variables
\begin{equation}\label{sub1}
x_{1}=\frac{f_{1}(1,w)^2}{f_{2}(1,w)}r,\quad x_{i}=w_{i}x_{1},\quad
y_{1}=pf_{1}(1,w)x_{1}^{m},\quad y_{2}=qf_{1}(1,w)x_{1}^{m},
\end{equation}
for $i=2,\ldots, n$, where $w=(w_{2},\ldots,w_{n})$. The inverse mapping is given by
\begin{equation*}
w_{i}=\frac{x_{i}}{x_{1}},\quad p=\frac{y_{1}}{f_{1}(\overline{X})},\quad
q=\frac{y_{2}}{f_{1}(\overline{X})}\quad
r=\frac{f_{2}(\overline{X})}{f_{1}(\overline{X})^2},
\end{equation*}
for $i=2,\ldots, n$.
The substitution given by (\ref{sub1}) leads to the surface
$\cal{W}'$ given by the equation (after clearing the common factor
$f_{1}(1,w)^4r^{4m}$)
\begin{equation*}\label{eq2}
\cal{W}':\;a(p^4-1)=b(q^4-r^2).
\end{equation*}
The transformation shows that $\cal{W}$ is just a cone over the surface $\cal{W}'$.
It is well known that $\cal{W}'$ is a del Pezzo surface of degree two with known
rational point $(1,1,1)$ and is unirational over $\Q$. In particular, this implies
the existence of a rational map of the following form
\begin{equation*}
\phi:\Q^2\ni (u,v) \mapsto (p,q,r)=(g_{1}(u,v),g_{2}(u,v),g_{3}(u,v))\in \cal{W}'.
\end{equation*}
The coordinates of $\phi$ are given by
\begin{equation*}
g_{1}(u,v)=\frac{bu^2(u^2-4uv+2v^2)+a}{bu^2(u^2-2v^2)+a},\; g_{2}(u,v)=\frac{bu^2(u^3-2u^2v+2uv^2)+a(u-2v)}{bu^2(u^2-2v^2)+a}
\end{equation*}
and $g_{3}(u,v)=u^2(g_{1}(u,v)^2+1)-g_{2}(u,v)^2$. In particular
$\overline{\phi(\Q)}=\cal{W}'(\R)$ in the Zariski topology and this property
immediately implies the density (in the Zariski topology) of rational points on
the hypersurface $\cal{W}$. The existence of a map $\phi$ implies also the existence
of a rational two-parametric solution of the equation defining $\cal{W}$.
\end{proof}

\section{Some additional remarks on $ax^2+by^{2p}=cz^{2q}+dw^{2r}$ with
$\frac{1}{p}+\frac{1}{q}+\frac{1}{r}=1$ and $abcd$ square}\label{sec7}

The equations of Sections 2 to 4 are special cases of the following;
\begin{equation*}
\cal{S}:\;ax^2+by^{2p}=cz^{2q}+dw^{2r},\quad \frac{1}{p}+\frac{1}{q}+\frac{1}{r}=1,
\end{equation*}
where $a, b, c, d\in\N$, $abcd=\square$. The triples of positive integers
$(p, q, r)$ satisfying the condition $\frac{1}{p}+\frac{1}{q}+\frac{1}{r}=1$, $p \leq q \leq r$, are precisely $(2,3,6), (2,4,4), (3,3,3)$.
In this section we make some remarks concerning the existence of rational points
on $\mathcal{S}$.  The main observation is the following.

\begin{thm}\label{generalabcd}
The surface $\cal{S}$ is birationally equivalent with a genus one curve defined
over the field of rational functions $\Q(t)$. In the case $(p,q,r)=(2,3,6)$ the
curve can be given in Weierstrass form.
\end{thm}
\begin{proof}
Recall (see for example Richmond~\cite{Ric}) that a quadratic form
$aX^2+bY^2-cZ^2-dW^2$, with $abcd=\square$, and possessing a non-zero rational
point, can be written in the form $L_{1}L_{2}-L_{3}L_{4}$ where $L_{i}$ is a linear form
in the variables $X,Y,Z,W$ with rational coefficients, for $i=1,2,3,4$.
This immediately implies that the equation defining the surface $\cal{S}$ can be
written in the form
\begin{equation*}
\cal{S}:\;L_{1}(x,y^p,z^q,w^r)L_{2}(x,y^p,z^q,w^r)=L_{3}(x,y^p,z^q,w^r)L_{4}(x,y^p,z^q,w^r).
\end{equation*}
Instead, we view $\mathcal{S}$ as a curve $\mathcal{C}$ defined over $\Q(t)$ by
the following intersection:
\begin{equation*}
\cal{C}:\;L_{1}(x,y^p,z^q,w^r)=t \;L_{3}(x,y^p,z^q,w^r), \quad t\; L_{2}(x,y^p,z^q,w^r)=L_{4}(x,y^p,z^q,w^r).
\end{equation*}
Eliminating $x$ (these are linear equations) there results an equation of the form:
\begin{equation*}
\cal{C}:\;A(t)y^p+B(t)z^q+C(t)w^r=0,
\end{equation*}
where $A,B,C\in\Z[t]$ depend on $a,b,c,d$. Because $(p,q,r)\in\{(2,3,6), (2,4,4), (3,3,3)\}$,
it follows immediately that $\cal{C}$ is of genus 1. If $(p,q,r)=(2,3,6)$ then $\cal{C}$ is
birationally equivalent to the elliptic curve in Weierstrass form $Y^2=X^3+A(t)^3B(t)^2C(t)$,
where $Y=A(t)^2B(t)\frac{y}{w^3}$ and $X=A(t)B(t)\frac{z}{w^2}$.
\end{proof}

\begin{exam}
{\rm
Consider the equation $x^2+y^6=2(z^6+w^6)$. It is readily checked that
\begin{align*}
(3X-Y-2Z-4W)(7X-Y-10Z)- & (3X+Y-4Z-2W)(5X-5Y-8Z-6W) \\
                        & = 6(X^2+Y^2-2Z^2-2W^2).
\end{align*}
Thus $x^2+y^6=2(z^6+w^6)$ if and only if there exists a rational number $t$ such that
\begin{equation*}
(3x+y^3-4z^3-2w^3)=t(3x-y^3-2z^3-4w^3),\;t(5x-5y^3-8z^3-6w^3)=7x-y^3-10z^3.
\end{equation*}
It follows that
\begin{equation*}
x=\frac{(t+1)y^3+2(t-2)z^3+2(2t-1)w^3}{3(t-1)},
\end{equation*}
and
\begin{equation*}
C_{t}:\;(5t^2-8t+5)y^3+(7t^2-10t+1)z^3+(-t^2+10t-7)w^3=0.
\end{equation*}
A small numerical search reveals that when $t=1/13$, the cubic curve has rational
point $P(y,z,w)=(5,18,7)$ corresponding to $x=8261$. It is easily checked that $P$ is
of infinite order on $C_{1/13}$, and thus the set of rational points on the surface
$x^2+y^6=2(z^6+w^6)$ is infinite.
}
\end{exam}

\begin{rem}
{\rm We undertook a small numerical search for rational points on the surface
$ax^2+by^6=cz^6+dw^6$ with $1\leq a,b,c,d\leq 5$ and $c\leq d$. For all but one
case in this range we found that the equation is insolvable modulo 3, or has
a solution with small height ($\leq 100$). The only case resisting this attack
is the equation $2x^2+y^6=2z^6+4w^6$. This equation can be written in the
alternative form $2(x^2-z^6)=4w^6-y^6$ and thus contains the pencil of cubic curves
\begin{equation*}
C_{t}:\;2(t^2-2)w^3-(t^2+2)y^3+4tz^3=0.
\end{equation*}
We used the Magma procedure {\tt PointsCubicModel} for $t$ in the range $H(t)\leq 100$
in order to find curves $C_{t}$ containing rational points with $\op{max}\{|x|,|y|,|z|\}\leq 10^6$.
However, no solution was found in this range.
}
\end{rem}

\noindent
{\bf Acknowledgments}. The authors thank the referee for a careful reading of the paper,
and for suggesting numerous improvements.
The first author acknowledges with gratitude the hospitality of
the Jagiellonian University, Krak\'ow, for a short visit when the results presented in this paper
were finalized; research of the second author was supported by Polish Government funds
for science, grant IP 2011 057671 for the years 2012--2013.

\bigskip

\noindent Andrew Bremner, School of Mathematical and Statistical Sciences,
Arizona State University, Tempe AZ 85287-1804, USA; e-mail:
bremner@asu.edu

\bigskip

\noindent Maciej Ulas, Jagiellonian University, Faculty of Mathematics and Computer Science, Institute of
Mathematics, {\L}ojasiewicza 6, 30-348 Krak\'ow, Poland; email:
Maciej.Ulas@im.uj.edu.pl

\end{document}